\documentclass{amsart}
\usepackage{amsmath}%
\usepackage{amsfonts}%
\usepackage{amssymb}%
\usepackage{graphicx}
\usepackage{hyperref}
\usepackage{mathtools}

\newtheorem{theorem}{Theorem}

\newtheorem{corollary}{Corollary}

\newtheorem{lemma}{Lemma}

\newcommand{\wrd}{W^r_d}

\newcommand{\nd	}{\mbox{ and }}
\newcommand{\cliff}{\mbox{Cliff}}
\newcommand{\mco}{\mathcal{O}}

\newcommand{\mck}{\mathcal{K}}
\newcommand{\coker}{\mbox{coker}}
\newcommand{\Hom}{\mbox{Hom}}

\begin{document}

\title{A Shorter Proof of Marten's Theorem}
\author{Adam Ginensky}
\date{6-2016}
\maketitle

\begin{abstract}
This note is dedicated to the memory of my friend and teacher R. Narasimhan. Marten's theorem is only one of the thousands of topics that Narasimhan could, and would, talk about with great enthusiasm and  with great depth of knowledge.
\end{abstract}

\tableofcontents

\section{Introduction and Notation	}

This note is intended to give a new proof on Marten's theorem stating that  $\dim(\wrd) \leq d-2r$ for any smooth curve with equality occurring exactly in the case when C is hyper-elliptic.  The proof follows the general lines of the proof given in 'The Geometry of Algebraic Curves'. What is different is that it uses Hopf's theorem to simplify the proof and strengthen the conclusions of the theorem. More specifically Mumford and Keem have explicated the curves for which  $\dim(\wrd) = d-2r-1$ . Our analysis allows us to state for any $ e>0$  which curves are eligible to have the bound $ \dim(\wrd) = d-2r-e$.  This is done in terms of the Clifford index of C. 

Finally, we sketch a proof of Hopf's Theorem.  In \cite{ACGH} this result is credited to Hopf, but no proof or reference is given. In  \url{http://mathoverflow.net/questions/156674/who-stated-and-proved-the-hopf-lemma-on-bilinear-maps} some references are given, but my reading of the references do not show a proof that an algebraic geometer might easily follow. Prof. Mohan Kumar has communicated to me an algebro-geometric proof of this result.  With his permission it is included in this note. I would also like to thank Prof. Kumar for reading an earlier version of this note and giving me corrections and useful comments. 

Mathematics has always been an avocation and not a vocation for me.  My lack of daily contact with mathematics and mathematicians makes me particularly prone to make errors and have slips of form.  For that reason all comments and corrections are greatly appreciated.  
\vspace{3mm}

\vspace{3mm}

We assume the following throughout-
\begin{itemize}
	\item C will be a smooth curve defined over an algebraically closed field. 
	\item L a special line bundle on C
	\item $\deg(L) = d ,\nd d \leq g-2 \nd h^0(L) = r+1 $
	\item that is  $L \in W^r_d$
	
\end{itemize}

 I will denote by $\alpha_L$ the Petri map 
\[ H^0(L) \otimes H^0(K_C \otimes L^{-1})\stackrel{\alpha_L}{\rightarrow}  H^0(K_C) . \]  
Notice that this map is always injective on each factor.  
In addition recall that  If $ L \cong \mathcal{O}_C(D)$, then $\cliff(L) = \cliff(D) = d-2r$ .
Finally we recall the basic result due to Hopf and it's corollary.\\

\begin{lemma}[Hopf] 
\label{hopfLemma}

Let $A\otimes B \to C$ be a bi-linear map of vector spaces defined over an algebraically closed field with $\dim(A) = a \nd \dim(B) = b$ .  If the map is injective on each factor, then $\dim(Im(A \otimes B) \geq a+b - 1$
\end{lemma}
%
 %

\begin{corollary}
L as above,

\begin{itemize}
	\item $\dim(\alpha_L)  \geq g - \cliff (L) $
	\item $\dim(\coker(\alpha_L)) \leq \cliff(L)$
\end{itemize}
\end{corollary}

\begin{proof}   Since $\alpha_L$ is injective on each factor and $(r+1) + (g-d+r) = g -(d-2r) +1 $, the first item follows from Hopf's lemma with $A = H^0(L), B = H^0(K_C \otimes L^{-1}) , \nd C = H^0(K_C)$.  The second assertion follows directly from the first assertion. 
\end{proof}  

\section{Marten's Theorem }

Marten's theorem bounds the dimension of $W_r^d$ .  We follow the proof in \cite{ACGH} p.192.   . Starting with the fact, 
\[ \dim(W_d^{r+1}) \leq \dim(T_{L,W_d^{r+1}}) = \dim(\coker(\alpha_L)) \]

 for a general $ L \in \wrd$, the authors then show  that $\dim(W_d^{r+1}) > d-2r$ leads to $h^0(K_C \otimes L^{-2})$ being too large to satisfy the conditions of Clifford's theorem .  We present a shorter and more direct proof using the corollary to Hopf's Lemma.

\begin{theorem}[Marten] $\dim(W_r^d) \leq d-2r$ and if C is not hyperelliptic, then $\dim(W_r^d) \leq d-2r-1$  
\end{theorem}

	\begin{proof} 
	As remarked above, it is enough to bound the tangent space, $T_{L,W_d^{r}}$ at a point $ L \notin T_{L,W_d^{r+1}}$.  By proposition 4.2 on p .189 of \cite{ACGH} we know that if $L \notin W_c^{r+1}$, then  \[ \dim(T_{L,W_d^{r}}) = \dim(\coker(\alpha_L ) )\] and by the corollary to Hopf's theorem, 
	\[ \dim(\coker(\alpha_L ) \leq \cliff(L) = d-2r  \] This bound is clearly achieved for hyperelliptic curves  The better bound in the case that C is not hyperelliptic will be explained and generalized in the next section. 
	\end{proof}

 \section{Extensions of Marten's Theorem}

The sharper result that in fact $\dim(W_r^d) \leq d-2r-1$ if C is not hyperelliptic follows from an inductive argument that bounds from below $\dim(h^0(K_C \otimes L^{-2}))$ .  Mumford (and Keem) have complimented the analysis by analyzing the border cases when $\dim(W_r^d) = d-2r-1$. We will present an argument below that we believe simplifies the exposition in \cite{ACGH}. The proof allows us to deduce a corollary that relates the dimension of $W_r^d $ to the Clifford index of the curve.  While not as sharp as Mumford's result for the border case  $\dim(W_r^d) = d-2r-1$, it does give new information as to which curves can achieve the bound $\dim(W_r^d) = d-2r-e$ with $ e \geq 2$ . The result is. \vspace{ 3mm}

\textbf{ Main Result:}  If $ \cliff(C) = c$ , then $\dim(W_r^d) \leq d-2r +\left\lceil \frac{c}{2}\right\rceil$. \vspace{ 3mm}

  Firstly, we may assume that L is base point free - else remove the basepoints and perform the argument with the new d. We introduce some notion.  Let ${s_1,s_2, \ldots, s_{r+1} }$ be a basis of $H^0(L)$, such that ${s_1,s_2 }$ span L.  Then for every $k \geq 2$ we get an exact sequence,
	
		\[ 0 \to \mbox{Ker}_k \to \mco^{k} \xrightarrow{(s_1,\ldots , s_k)}  L \to 0
	\]

 and after tensoring with $K_C \otimes L^{-1}$ we get another exact sequence-

 \[  0 \to \mck_k \to  \left((K_C \otimes L^{-1}) \otimes \mco^{k}\right) \stackrel{\alpha_k}  \to K_C \to 0  \] 
where $\mck_k =\mbox{Ker}_k \otimes K_C \otimes L^{-1}$ and $\alpha_k$ is the Petri map restricted to the k given sections. 
By the base point free pencil trick , $\mck_2 \cong K_C \otimes L^{-2} $ and hence  for $k>2$, we have an exact sequence .

\[ 0 \to K_C \otimes L^{-2} \to \mck_{k} \to \left((K_C \otimes L^{-1}) \otimes \mco^{k-2}\right)  \to 0 \]  

Setting $ k = r+1 $ we get an inequality, 
	 \[ h^0(\mck_{r+1}) \leq h^0\left(K_C \otimes L^{-1}\right) + (r-1)h^0(K_C \otimes L^{-2}) = h^0\left(K_C \otimes L^{-1}\right) + (r-1)(g-d+r)    
 \]
If  $\dim((\alpha_L)) = g - \cliff(L) + e $ , then,   
	\[  h^0(\mck_{r+1})  = \dim(\ker(\alpha_L)) = (g-d+r)(r+1) - (g-d + 2r +e),
\]
 
and consequently we can bound $h^0(K_C\otimes L^{-2}) $ from below by (and this is the essence of the argument in \cite{ACGH}) via 
\begin{align*}
	h^0(K_C\otimes L^{-2})  \geq &\dim(h^0(\mck_{r+1})) - (r-1)h^0(K_C \otimes L^{-1})\\
	& = \dim(h^0(\mck_{r+1})) - (r-1)(g-d+r)\\
	& = \left[h^0\left((K_C \otimes L^{-1}) \otimes \mco^{r+1}\right) -\dim(\alpha_L)\right] - (r-1) (g-d+r) \\
  & =(g-d+r)(r+1) - (g-(d - 2r) +e) - (r-1)(g-d+r) \\
	& = g -d -e 
 \end{align*}

	 Since 
		 \[ \deg(K_C\otimes L^{-2}) = 2g-2 - 2d \nd h^0(K_C\otimes L^{-2}) \geq g-d -e
	 \]
		we compute that  
			\[ \cliff(K_C\otimes L^{-2}) = \cliff(\mco(2D) \geq 2g-2-2d - 2(g-d+e-1)  = 2e
		\]
		 The conclusion is, since by assumption $2D$ is eligible to calculate the Clifford index, that $ 2e \geq \cliff(C)$ .  We summarize the discussion in \\
		
		\begin{theorem} Suppose $\dim(\alpha_L) = g-d+2r +e$ , then  $ 2e \geq \cliff(C)   $ and consequently $ e \geq \left\lceil \frac{c}{2}\right\rceil$
		\end{theorem}

\begin{corollary} Suppose $\cliff(C) = c, $ then for any special line bundle L, we have $\dim(\alpha_L) \geq g- (d-2r) + \left\lceil \frac{c}{2}\right\rceil  $
\end{corollary}
\begin{proof} Our assumption is that  $\dim((\alpha_L)) = g - \cliff(L) + e $ and $ e \geq \left\lceil \frac{c}{2}\right\rceil $ by the previous corollary.     
\end{proof}

\begin{corollary}[Main Result]  $\dim(W_r^d) \leq d-2r -\left\lceil \frac{c}{2}\right\rceil$.  In particular, if C is not hyperelliptic, $\dim(W_r^d) \leq d-2r-1
  $ and if $\dim(W_r^d) = d-2r-1 $, then the Clifford index of C is atmost two.
	\end{corollary}

\section{ Discussion of Mumford's refinement of Marten's Theorem }

The last corollary in some sense generalizes Mumford's refinement.  It is not as precise as Mumford's result for the cases that Mumford covers, but it does give different information.  What Mumford proves is 
\begin{theorem}[Mumford] If $ \dim(\wrd) = d -2r-1 $ then C is either trigonal, bi-elliptic, or a smooth plane quintic.
\end{theorem}

Notice that in trigonal or plane quintic are exactly the cases of Clifford index one and the a bi-elliptic curve does have Clifford index 2.  However there are many curves of clifford index 2 which are not bi-elliptic .  In particular a plane sextic has Clifford index 2.

The new information of this note suggests an approach  to recreating and extending Mumford's result.  Namely to find the curves for which the bound in Marten's theorem is sharp, that is to find the curves for which  $ \dim(\wrd) = d -2r-\left\lceil \frac{c}{2}\right\rceil$ , one has to find curves on which there is a divisor D such that not only is Cliff(D) 'small', but so is Cliff(2D). 
 I suspect that, in the end, the kind of case by case analysis that Mumford did is necessary.  

On the other hand, Corollary 3 gives more precise information than was previously reported in the literature on how the bound in Marten's theorem is dependent on the existence of very  special divisors. To the best of my knowledge, the Corollary provides new information about the case when $ \dim (\wrd) \leq d-2r-2. $ The best result (known to me) is Keem's theorem-
\begin{theorem}[Keem] If C is a smooth algebraic curve of genus $g \geq 11$  and we have integers d and r satisfying $ d \leq g+r-4$ and  we have $ \dim(\wrd) \geq d - 2r -2 $ then C has a $g^1_4$ and hence is of Clifford index 2.
\end{theorem} 
Notice that the genus of a plane sextic has genus 10 and that is why (presumably) all curves of lower genus are eliminated.  By comparison, Corollary 2 gives- 

\begin{theorem} Suppose C is a smooth curve , $ d \leq g-2 $ and   $ \dim(\wrd) = d - 2r -2 $, then the Clifford index of C is two and hence C is either a plane sextic or 4 gonal.
\end{theorem} 
\begin{proof} Every curve of Clifford index two is either a plane sextic or possesses a $g^1_4$ .
\end{proof} 

This improves upon Keem's result by eliminating the lower bound on the genus and improving the bound on r to the optimal value.  

\section{ Proof of Hopf's Lemma}

We wish to thank Prof. Mohan Kumar for allowing us to include this proof.  While it is conceivable that this proof is known to others, I can find no reference to it in the literature.  I would sincerely appreciate any comments or references on this topic.  I note that while this result is called Hopf's theorem in 'The Geometry of Algebraic Curves', it seems also to be known as Hopf's lemma or Hopf's lemma on bi-linear forms in the topology literature.  We call it a lemma in deference to it's usage in topology, which seems to have been the original use of the lemma.
\begin{lemma}[Hopf] 

Let $A\otimes B \to C$ be a bi-linear map of vector spaces defined over an algebraically closed field with $\dim(A) = a \nd \dim(B) = b$ .  If the map is injective on each factor, then $\dim(Im(A \otimes B) \geq a+b - 1$\end{lemma}

\begin{proof}

The proof uses the fact that if we let $M=(x_{ij})$ be the space of matrices of size m x n, and set $X_r \subset \mathbb{A}^{mn} $ to be all the matrices of rank at most r, then $X_r$ has codimension equal to $(m-r)(n-r)$ in $\mathbb{A}^{mn} $.  While Prof. Kumar communicated a very short algebraic proof of this fact, in the spirit of this note we prefer to reference \cite{ACGH}  p. 67. 

Firstly, note we may trivially assume that $ b \geq 2$  and that further  $\dim(C)= c = a+b-2 $ . 
Denote by $f_b$ the map $ A \to C $ given by $\otimes b$ with $ b \in B $.  By assumption, all the maps $f_b$ are injective.  In fact, if  $b_1, \ldots b_n$ are a basis for B, then, setting $ f_1, \ldots f_n$ to be the linear maps corresponding to the $b_i$, then the $f_i$   must be linearly independent as elements of  $ \mathbb{P}\left(\Hom\left(A,C\right)\right) $ because any linear relation $ \sum a_i f_i = 0 $ would imply that $ f_{\sum a_i b_i} = 0$ . Consequently these elements span a linear space $ H \subset \mathbb{P}^{ac-1} $ of dimension b-1.  Consider $ X_1 \subset \mathbb{P}^{ac-1} $ The subset of maps $A \to C $ which are not injective. This has codimension 
	\[ (c-(a-1) )(a - (a-1) ) = c - a + 1 = a+b-2 -a + 1 = b-1 
\]
Hence $ X_1 \cap H \neq \emptyset $. Since we have assume our base field is algebraically closed this means there exists a point $ p \in X_1 \cap H$ with coordinates $(a_1, \ldots a_n) $ and then $ \sum a_i f_i $ is not injective , contradicting the hypothesis of the lemma .

\end{proof}

\bibliographystyle{alpha}

\end{document}